\documentclass[10pt]{amsart}

\usepackage{graphicx}

\allowdisplaybreaks
\usepackage{enumerate} 
\usepackage{mathrsfs}

\newtheorem{theorem}{Theorem}[section]
\newtheorem{lemma}[theorem]{Lemma}

\theoremstyle{definition}

\theoremstyle{remark}
\newtheorem{remark}[theorem]{Remark}

\numberwithin{equation}{section}

\newcommand{\calL}{\mathcal{L}}
\newcommand{\calR}{\mathcal{R}}

\newcommand{\scrA}{\mathscr{A}}
\newcommand{\scrD}{\mathscr{D}}
\newcommand{\scrW}{\mathscr{W}}

\DeclareMathOperator{\Sym}{Sym}
\DeclareMathOperator{\pr}{pr}
\DeclareMathOperator{\Span}{span}
\DeclareMathOperator{\tr}{tr}

\DeclareMathOperator{\dv}{div}
\DeclareMathOperator{\Ric}{Ric}

\DeclareMathOperator{\SO}{SO}
\DeclareMathOperator{\SU}{SU}

\newcommand{\ve}{\varepsilon}

\begin{document}

\title[A Lichnerowicz Estimate for the Sub-Laplacian]{A Lichnerowicz Estimate for the Spectral Gap of the Sub-Laplacian}


\author[S.~M.~Berge]{Stine Marie Berge}
\address{Department of Mathematical Sciences, Norwegian University of Science and Technology, 7491 Trondheim, Norway}
\email{stine.m.berge@ntnu.no}

\author[E.~Grong]{Erlend Grong}
\address{Universit\'e Paris Sud, Laboratoire des Signaux et Syst\`emes (L2S) Sup\'elec, CNRS, Universit\'e Paris-Saclay, 3 rue Joliot-Curie, 91192
Gif-sur-Yvette, France and University of Bergen, Department of Mathematics, P. O. Box 7803,
5020 Bergen, Norway.}
\email{erlend.grong@gmail.com}
\thanks{The second author is supported by the Research Council of Norway (project number 249980/F20).
 The authors were partially supported by the joint NFR-DAAD project 267630/F10.
Results are partially based on the first authors Master Thesis at the University of Bergen, Norway.}


\date{}

\dedicatory{}


\begin{abstract}
For a second order operator on a compact manifold satisfying the strong H\"ormander condition, we give a bound for the spectral gap analogous to the Lichnerowicz estimate for the Laplacian of a Riemannian manifold. We consider a wide class of such operators which includes horizontal lifts of the Laplacian on Riemannian submersions with minimal leaves.
\end{abstract}

\maketitle


\section{Introduction}
Consider a second order operator $L = \sum_{i=1}^k X_i^2 +X_0$ of strong H\"ormander type defined on a compact manifold $M$, symmetric with respect to a smooth volume density~$\mu$. In other words, we assume that $X_1, \dots, X_k$ along with their iterated brackets span the entire tangent bundle, giving us that that $L$ is hypoelliptic with discrete spectrum~\cite{Hor67}, see also e.g. \cite{FePh83}. The first general result regarding the spectral gap of $L$ was given in \cite{Jer86} using coordinate expression of the vector fields $X_1, \dots, X_k$. Several results exist for the special case when $L$ is the sub-Laplacian of a CR-manifold, see e.g. \cite{Gr85,LiWa13,IPV14}. More general coordinate-free results exist in \cite[Section~2]{BaBo12}, \cite[Section~3.5]{BaKi14}, \cite{BaKi16} and \cite[Section~4.1]{GrTh16a}. A majority of the latter mentioned group of results concern generalizations of Lichnerowicz bound for the spectral gap \cite{Lic58}. More precisely, if $L$ is the Laplacian of a compact, $n$-dimensional Riemannian manifold with Ricci curvature bounded from below by $\rho > 0$, then the first non-zero eigenvalue $\lambda_1$ has the lower bound
$$|\lambda_1| \ge \frac{n}{n-1} \rho.$$

In this paper, we will present a further generalization of the Lichnerowicz spectral gap. In order to put the result in perspective, consider the following special case. Let $\pi: (M, g) \to (N, g_N)$ be a surjective Riemannian submersion, meaning that the orthogonal complement $H$ of the vertical bundle $V = \ker d\pi$ satisfies
\begin{equation} \label{RieSub} g| H = \pi^* g_N |H.\end{equation}
Consider the operator $L$ defined to be the horizontal lift with respect to $H$ of the Laplacian on $N$. Then $L$ is symmetric with respect to the volume density $\mu_g$ of $g$ if and only if the submanifolds $M_y = \pi^{-1}(y)$, $y \in N$ are minimal, see \cite[Section~2]{GrTh16a} for more details. If in addition these submanifolds are totally geodesic and $H$ satisfies the Yang-Mills condition, a spectral gap for $L$ is given in \cite{BaKi16}. This estimate is sharp for the cases of the Hopf fibration and the quaternionic Hopf fibration. We will generalize this result, removing the requirement of the Yang-Mills condition and only requiring that the leaves are minimal (Condition \eqref{item:C}, Section~\ref{sec:Lich}). Furthermore, we do not need that the complement $V$ of $H$ is an integrable subbundle, allowing for the result to be applied outside the cases of submersions or foliations.

The main result with proof is presented in Section~\ref{sec:Main}. In Section~\ref{sec:Examples} we give examples where the estimate on the spectral gap is computed explicitly. We emphasize the fact that the removal of the requirement of the Yang-Mills condition allows us to consider conformal changes of the metric $g_N$ in \eqref{RieSub}, see Section~\ref{sec:YM}.

\section{Statement of main result and proof of the spectral gap} \label{sec:Main}
\subsection{Sub-Laplacians and notation}
Let $M$ be a connected, compact manifold with a volume density $\mu$. Let $L = \sum_{i=1}^k X_i^2 +X_0$ be a smooth second order operator on $M$ that is symmetric with respect to $\mu$. Any such operator $L$ is uniquely determined by a symmetric positive semi-definite tensor $q_L \in \Gamma(\Sym^2 TM)$ on the cotangent bundle, given by
$$L(f_1 f_2) - f_2 L f_1 - f_1 L f_2  = 2 \langle df_1, df_2 \rangle_{q_L}.$$

\emph{A sub-Riemannian structure} on $M$ is a pair $(H, g_H)$ where $H$ is a subbundle of the tangent bundle $TM$ and $g_H$ is a fiber metric defined only on $H$. We will call $H$ \emph{the horizontal bundle}. Equivalently, a sub-Riemannian structure can be considered as a metric tensor $g_H^*$ on $T^*M$ that degenerates along a subbundle. These two points of view are related through the map
$$\begin{array}{ccc} \sharp^H:T^*M \to TM, & \qquad & \mathrm{Image} \, \sharp^H = H \vspace{0.2cm} \\ 
 \sharp^H \alpha = \langle \alpha, \, \cdot \, \rangle_{g_H^*}, & \qquad & \langle \alpha, \beta \rangle_{g_H^*} = \langle \sharp^H \alpha, \sharp^H \beta \rangle_{g_H} . \end{array}$$

Any operator $L$ where $q_L = g_H^*$ degenerates along a subbundle can be considered as \emph{the sub-Laplacian of $(M, H, g)$ with respect to $\mu$}, defined by
$$Lf = \Delta_{H,\mu} f := \dv_\mu \sharp^H df.$$
For the special case when $H = TM$, $\Delta_{TM, \mu}$ is called the Witten-Laplacian or just the Laplacian if $\mu$ is the Riemannian volume density.

In what follows, we assume that $H$ is \emph{bracket-generating}, meaning that the vector fields with values in $H$ and their iterated brackets span the entire tangent bundle. This is equivalent to assuming that $L = \Delta_{H,\mu}$ satisfies the strong H\"ormander condition. The operator $L$ is essentially self-adjoint on smooth functions \cite{Str86} and the main result will give a bound for the first non-zero eigenvalue of $\Delta_{H,\mu}$ using bounds on tensors. 

For the rest of the paper, we will use the following notation. The manifold $M$ will be connected, compact and of dimension $m+n$, where $n$ denotes the rank of the horizontal bundle $H$. For any 2-tensor $\xi \in T^*M^{\otimes 2}$, define $\tr_H \xi(\times, \times) = \xi(g_H^*)$. In other words, $\tr_H \xi(\times, \times)(x) = \sum_{i=1}^n \xi(v_i, v_i)$ where $v_1, \dots, v_n$ is an orthonormal basis of $H_x$.

For a given volume density $\mu$ we will write $\| f\|_{L^2}$ for the corresponding $L^2$ metric. If $q$ is a possibly degenerate metric tensor on a vector bundle $E \to M$, we write $| e|_q = \langle e, e \rangle_q^{1/2}$ for any $e \in E$, and $\| Z\|_{L^2(q)} := \| \, | Z |_q \, \|_{L^2}$ for any $Z \in \Gamma(E)$. Finally, we introduce the following operators on forms. If $T: \bigwedge^k TM \to TM$ is a vector-valued $k$-form, define an operator $\iota_T$ such that
$$\iota_T \alpha = \alpha(T(\, \cdot \,)) \quad \text{if $\alpha$ is a one-form},$$
and extend its definition to arbitrary forms by requiring that it satisfies $\iota_T (\alpha \wedge \beta) = (\iota_T \alpha) \wedge \beta + (-1)^{k + \deg \alpha} \alpha \wedge \iota_T \beta$. Note that $\iota_T$ becomes the usual contraction when $T$ is a vector field, i.e. $k = 0$.

\subsection{A Lichnerowicz estimate} \label{sec:Lich}
Let $( M,H,g_H )$ be a sub-Riemannian manifold. We introduce the following assumptions.
\begin{enumerate}[\rm (A)]
\item \label{item:A} There exists a complement $V$, i.e. a choice of subbundle satisfying $TM=H\oplus V$, such that $g_H$ and the corresponding projection $\pr_H: TM \to H$ satisfy
\begin{equation}\label{mp}
( \calL_Z\pr_{H}^*g_H )( X,X )=0, \qquad \text{for $Z\in \Gamma(V)$ and $X\in \Gamma(H)$,}
\end{equation}
with $\calL$ denoting the Lie derivative.
In this case $V$ is called \emph{a metric preserving complement} of $(H, g_H)$.
\item \label{item:B} With respect $V$, define \emph{the curvature} and \emph{the cocurvature} of $H$, by
$$\mathcal{R}( X,Y )=\pr_{V}\left[ \pr_{H}X,\pr_{H}Y \right] \quad  \text{ and }  \quad \bar{\mathcal{R}}( X,Y )=\pr_{H}\left[ \pr_{V}X,\pr_{V} Y\right],$$ respectively. We will assume that 
\begin{equation}\label{eq:13}
\tr\bar{\mathcal{R}}( X,\mathcal{R}( X, \, \cdot \, ) ) =0.
\end{equation}
This condition has appeared in \cite{GrTh16}, \cite{GrTh16a} and \cite{GrTh16b}. Note that if $V$ is integrable then $\bar{\mathcal{R}}\equiv 0$, hence \eqref{eq:13} is always satisfied.
\item \label{item:C} A Riemannian metric $g$ is said \emph{to tame} $g_H$ if $g|H = g_H$. Let $g$ be a taming metric for $g_H$ making $H$ and $V$ orthogonal. Define $g_{V}$ to be the restriction of $g$ to $V$. It will be assumed that 
\begin{equation}\label{eq:16}
   \tr_{V}( \calL_{X}g )( \times,\times )=0, \qquad \text{whenever $X\in \Gamma( H )$.}
\end{equation}
\end{enumerate}
Let $\nabla^V$ be any affine connection on $V$ compatible with $g_{V}$, and denote the Levi-Civita connection of $g$ by $\nabla^{g}$. Define the affine connection $\nabla$ on $TM$ by 
\begin{equation}\label{NablaDef}
\nabla_{X} Y= \pr_H \nabla^{g}_{\pr_{H}X}\pr_{H}Y+\pr_{H}\left[ \pr_{V}X,\pr_{H}Y \right]+\nabla^V_{X}\pr_{V}Y.
\end{equation}
Let $T = T^\nabla$ and $R^\nabla$ denote the torsion and curvature tensor of $\nabla$. In what follows we introduce the tensors playing a role in the estimate of the spectral gap.

Define the vector valued one-form $B: TM \to TM$ by
\begin{equation} \label{B} B(v) = \tr_H (\nabla_\times T)(\times, v) - \tr_H T(\times, T(\times, v)).\end{equation}
Furthermore, we introduce the endomorphisms $\Ric, \scrW : T^*M \to T^* M$ and the map $S:T^*M \to T^*M^{\otimes 2}$ by
\begin{eqnarray} \label{eq:Ric}
\Ric(\alpha)(v) & = & \tr_H R^\nabla(\times, v) \alpha(\times), \\
\langle \scrW(\alpha), \beta\rangle_{g^*} & = & \langle \iota_T \alpha, \iota_T \beta \rangle_{g_H^* \otimes g_V^*} - \langle \iota_B \alpha, \beta \rangle_{g_V^*},\\
\label{S} S(\alpha)(v,w) & =&  \langle w, T(\pr_H v, \sharp^H \alpha) \rangle_g 
\end{eqnarray}
for any $\alpha, \beta \in T^*M$ and $v,w \in TM$. We will use bounds on these tensors to obtain the estimate for the spectral gap.
\begin{theorem}\label{th:main}
Assume that $( M,H,g_H )$ is a compact sub-Riemannian manifold of rank $n$ with a taming metric $g$ such that the orthogonal complement $V$ is metric preserving, and assumptions \eqref{eq:13} and \eqref{eq:16} are satisfied. Let $\Delta_{H,\mu}$ be the sub-Laplacian with respect to the volume density $\mu$ of $g$. Denote by $\nabla$ any connection on the form of \eqref{NablaDef} satisfying
\begin{equation*}
\begin{aligned}
\langle \Ric( \alpha ),\alpha\rangle_{g_H^*} & \geq & \rho_{1} |\alpha|^{2}_{g_H^*}, \\
| \iota_T \alpha |_{g_H^{* \otimes 2}}^{2} & \geq & \rho_2 |\alpha|^2_{g_{V}^*},
\end{aligned} \qquad
\begin{aligned}
|S(\alpha) |^{2}_{g_H^*\otimes g_{V}^*} & \le & \kappa_1 |\alpha |^{2}_{g_H^*}, \\
|\langle \iota_B  \alpha,\alpha\rangle_{g_H^*} | & \le & 2\kappa_2 |\alpha |_{g_{V}^*} |\alpha|_{g_H^*}, \\
|\langle \scrW( \alpha),\alpha\rangle_{g^*}| &\le & \kappa_3 |\alpha|^{2}_{g_{V}^*},
\end{aligned}
\end{equation*} 
with $\rho_i > 0$ and $\kappa_j \geq 0$. Assume that
$$\rho_1 \rho_2 - 4 \kappa_2^2-3\kappa_1\kappa_3-8\kappa_2\sqrt{\kappa_1\kappa_3} > 0.$$
Then we have the following estimate on the spectral gap
$$-\lambda_{1}\ge \left(\sqrt{\frac{\rho_1 \rho_2 - 4\kappa_2^2  + \frac{n-1}{n} \rho_2 \kappa_3}{\frac{n-1}{n} \rho_2+ 3\kappa_1} + \left( \frac{4 \kappa_2 \sqrt{\kappa_1}}{\frac{n-1}{n} \rho_2 + 3 \kappa_1} \right)^2 } -\frac{4 \kappa_2 \sqrt{\kappa_1}}{\frac{n-1}{n}\rho_2  + 3\kappa_1} \right)^{2}-\kappa_3,$$
where $\lambda_1$ is the first non-zero eigenvalue of $\Delta_{H,\mu}$.
\end{theorem}

\begin{remark} \label{re:PostMain}
\begin{enumerate}[\rm (a)]
\item If $H =TM$, we can choose $\kappa_3 = 0$ and $\rho_2 = \infty$ to obtain the original Lichnerowicz estimate.
\item One can interpret a complement $V$ satisfying \eqref{mp} as a generalization of vertical bundles on sub-Riemannian manifolds coming from submersions. If $\pi: M \to N$ is a surjective submersion into a Riemannian manifold $(N, g_N)$, $V = \ker d\pi$ and $H$ satisfy $TM = H \oplus V$, then we can define a sub-Riemannian metric by $g_H = \pi^* g_N |H$. In this case, $V$ is a metric preserving complement and $\Ric$ is the pullback of the Ricci curvature on $N$, see~\cite[Prop~3.4]{GrTh16a}. More generally, if $V$ is an integrable and metric preserving complement, then the corresponding foliation of $V$ is a Riemannian foliation.
\item The assumption $\rho_2 > 0$ is equivalent to the assumption $H + [H,H] = TM$, i.e. that $\calR$ is surjective on $V$. To see this, observe that $T(v,w)= - \calR(v,w)$ whenever $v, w \in H$. Hence, $| \iota_T df |^2_{g_H^{*\otimes 2}} = | \iota_\calR df |^2_{g_H^{*\otimes 2}}$. It follows that the existence of a covector $\alpha$ such that $\iota_\mathcal{R}\alpha\equiv 0$ while satisfying $|\alpha |_{g_V^*} \neq 0$, implies $\rho_2 = 0$. Conversely, if $\calR$ is surjective, then $g_\calR^*(\alpha, \beta) := \langle \iota_\calR \alpha, \iota_\calR \beta \rangle_{g_H^{*\otimes 2}}$ is a cometric degenerate on $H$. Hence, it induces an inner product on $V$, and from compactness, we get that $| v |_{g_\calR}^2 \geq \rho_2 | v|_{g_V}^2$ for any $v \in V$ and some $\rho_2 >0$.
\item The above result is a generalization of a result found \cite{BaKi16}. The latter result was proved for the case when $V$ is an integrable subbundle, and the corresponding foliation of $V$ is \text{totally geodesic}, which is equivalent to the requirement
\begin{equation} \label{TGF} ( \calL_{X}g )(Z,Z )=0, \end{equation}
for any $X \in \Gamma(H)$ and $Z \in \Gamma(V)$. If this assumption holds, we can choose $\nabla$ as the Bott connection
\begin{align} \label{Bott}
\nabla_{X} Y & = \pr_H \nabla^{g}_{\pr_{H}X}\pr_{H}Y+\pr_{H}\left[ \pr_{V}X,\pr_{H}Y \right] \\ \nonumber
& \qquad+ \pr_V \nabla^{g}_{\pr_{V}X}\pr_{V}Y+\pr_{V}\left[ \pr_{H}X,\pr_{V}Y \right].
\end{align}
Under the additional assumptions of \emph{the Yang-Mills condition}, i.e. that \\$\tr_H (\nabla_\times T)(\times, \cdot ) = 0$, we get that $B =0$, $\scrW = 0$, and the result simplifies to
$$-\lambda_{1}\ge \frac{\rho_1 \rho_2  }{\frac{n-1}{n} \rho_2+ 3\kappa_1}  .$$
\item If the condition \eqref{TGF} holds, we can still use the connection \eqref{Bott}, even without the Yang-Mills assumption or even assuming $V$ integrable. From the fact that $T(H,V) = 0$ and assumption \eqref{eq:13}, we get $\scrW =0$ and hence the estimate
$$-\lambda_{1}\ge \left(\sqrt{\frac{\rho_1 \rho_2 - 4\kappa_2^2  }{\frac{n-1}{n} \rho_2+ 3\kappa_1} + \left( \frac{4 \kappa_2 \sqrt{\kappa_1}}{\frac{n-1}{n} \rho_2 + 3 \kappa_1} \right)^2 } -\frac{4 \kappa_2 \sqrt{\kappa_1}}{\frac{n-1}{n}\rho_2  + 3\kappa_1} \right)^{2}.$$
In this case, the result is non-trivial for $\rho_1 \rho_2 - 4\kappa_2^2 >0$. The latter assumption is analogous to positive Ricci curvature in the sense that if $g$ is a complete Riemannian metric and this inequality holds, then $M$ is compact \cite[Prop 5.3]{GrTh16b}.
\item Note that the result is invariant under scaling of the vertical part of the metric. Consider a variation of the metric $g_\ve = g_H \oplus \frac{1}{\ve} g_V$. If $g$ has volume density $\mu$, then clearly the density of $g_\ve$ is $\ve^{-m/2} \mu$. Hence, the sub-Laplacian with respect to $g$ and $g_\ve$ coincides. Applying this scaled metric to Theorem~\ref{th:main}, the parameters $\rho_1$, $\rho_2$, $\kappa_1$, $\kappa_2$ and $\kappa_3$ are respectively replaced by $\rho_1$, $\ve^{-1} \rho_2$, $\ve^{-1} \kappa_1$, $\ve^{-1/2} \kappa_2$ and $\kappa_3$, leaving us with the same result for the spectral gap.
\end{enumerate}
\end{remark}

\subsection{Properties of the connection}
The proof of the main theorem will be divided into several intermediate steps. We begin by proving some general properties of the connections used in this setting, given assumptions \eqref{item:A}, \eqref{item:B} and \eqref{item:C}.
\begin{lemma}
Let $( M,H,g_H )$ be a sub-Riemannian manifold and let $g$ be a taming metric such that the vertical bundle $V$ is metric preserving. If $\nabla$ is an affine connection defined as in \eqref{NablaDef}, then it satisfies
\begin{enumerate}[\rm (i)]
\item $\nabla g_H^* = 0$ and $\nabla g = 0$.
\end{enumerate}
Furthermore, we have that its torsion $T$ satisfies:
\begin{enumerate}[\rm (i)]
\setcounter{enumi}{1}
\item $\pr_H T(v, w) = 0$ whenever $v \in H$.
\item $T(v,w) = - \calR(v,w)$ whenever $v,w \in H$.
\item $\pr_H T(v,w) = - \bar{\calR}(v,w)$ whenever $v,w \in V$.
\item If $B$ is defined as in \eqref{B}, then it takes its values in $V$.
\item $\tr T(v, \, \cdot \,) = 0$ whenever $v \in H$.
\end{enumerate}
\end{lemma}
\begin{proof}
Properties (i) to (iv) follows from the definition of $\nabla$ and assumption \eqref{item:A}. The result (v) follows from (iii) and assumption \eqref{item:B}. For the proof of (vi), let $X$ be any section of $H$. Let $X_1,\dots,X_n$ and $Z_1,\dots,Z_m$ be local orthonormal bases of respectively $H$ and~$V$. Then we have that
\begin{align*}
& \tr T(X, \, \cdot \,) = \sum_{j=1}^n \langle T( X,X_j ), X_j\rangle_{g} + \sum_{k=1}^m \langle T( X,Z_k ),Z_k\rangle_{g} = \sum_{k=1}^m \langle T( X,Z_k ),Z_k\rangle_{g}\\
&=\sum_{k=1}^m \langle Z_k, \nabla_{X}Z_k-\nabla_{Z_k}X-\left[ X,Z_k \right]\rangle_{g} \\
& =- \sum_{k=1}^m \langle Z_k, [ X,Z_k ]\rangle_{g}=\frac{1}{2}\tr_{V}( \calL_{X}g )( \times,\times )=0,
\end{align*}
where the last equality follows from assumption \eqref{item:C}.
\end{proof}

\begin{lemma} \label{lemma:LambdaVertical}
Let $\Ric$ be defined as in equation \eqref{eq:Ric}
with respect to $\nabla$. Then 
$$\langle \Ric( \alpha),\beta\rangle_{g^*}=\langle \Ric( \alpha), \beta\rangle_{g_H^*}, \qquad \alpha, \beta \in T^*M.$$ Additionally, we have that $\Ric$ is symmetric.
\end{lemma}
\begin{proof}
Since $\nabla$ is compatible with $g$, we get $\langle R^{\nabla}( \, \cdot \, ,\, \cdot\, )v,v\rangle_{g}=0,$ for all $v \in TM$. Using the first Bianchi identity for affine connections with torsion, when $X,Y\in \Gamma( H )$ and $Z\in \Gamma( V )$,
\begin{align*}
&\langle R^\nabla( X,Z )Y + R^\nabla( Y,X )Z + R^\nabla( Z,Y )X,X \rangle_g= \langle R^\nabla( X,Z )Y,X \rangle_g\\
&= \langle - T( X,T( Z,Y ) )-T( Y,T( X,Z ) ) -T( Z,T( Y,X ) ),X\rangle_g\\
&\quad+\langle ( \nabla_{X}T )( Z,Y )+( \nabla_{Y}T )( X,Z )+( \nabla_{Z}T )( Y,X ),X \rangle_g = - \langle X, \bar{\calR}(Z, \calR(Y, X) \rangle_g.
\end{align*}
Hence 
\begin{align*}
& \left\langle \Ric( \alpha ),\beta\right\rangle_{g^*}
=\tr_{H}\left\langle R^{\nabla}( \times ,\sharp^{g}\beta) \sharp^{H}\alpha,\times\right\rangle_{g_H}\\
&=\tr_{H}\left\langle R^{\nabla}( \times ,\sharp^{H}\beta) \sharp^{H}\alpha,\times\right\rangle_{g_H} - \tr \bar{\calR}(\sharp^V \beta , \calR(\sharp^H \alpha, \, \cdot \, )) =\left\langle \Ric (\alpha),\beta\right\rangle_{g_H^{*}},
\end{align*}
by assumption \eqref{item:B}. Another application of the first Bianchi identity shows that $\langle R^\nabla(X,Y)Z,X\rangle_{g_H}=\langle R^\nabla( X,Z )Y,X\rangle_{g_H}$ when $X,Y,Z\in \Gamma( H )$, giving us that $\Ric$ is symmetric.
\end{proof}

\subsection{Necessary identities}
Fix a Riemannian metric $g$ taming $g_H$. For any connection $\nabla$, define the operators on tensors by $D[X,Y] = \flat X \otimes \nabla_Y$ and $L[X,Y] = \nabla^2_{X,Y} := \nabla_X \nabla_Y -\nabla_{\nabla_X Y}$ for two vector fields $X$ and $Y$. Extend these operators to arbitrary sections of $TM^{\otimes 2}$ by $C^\infty(M)$-linearity, and define $L= L[g_H^*]$ and $D = D[g_H^*]$. The following results are found in \cite{GrTh16}.
\begin{lemma}\label{SymmetryAndBochner}
Assume that $\nabla$ is compatible with $g$ and $g_H^*$.
\begin{enumerate}[\rm (a)]
\item We have $L = - D^* D$ on tensors if and only if
\begin{equation} \label{SymmetryCondition} \tr T(v, \, \cdot \, ) = 0, \qquad \text{for any $v \in H$.} \end{equation}
Here the dual is with respect to the $L^2$-inner product on tensors induced by~$g$. Also, in this case $D^* = - \tr_H \iota_{\times} \nabla_{\times}$. In particular, this is true for the connection $\nabla$ in \eqref{NablaDef} given assumption \eqref{item:C}.
\item If $\nabla$ satisfies \eqref{SymmetryCondition}, then for any $f \in C^\infty(M)$, we have $Lf = \Delta_{H,\mu}f$ with $\mu$ being the volume density of $g$, and furthermore
\begin{equation} \label{Weitz} Ldf-dLf= - 2\scrD df +\scrA( df),\end{equation}
\begin{eqnarray*} \label{scrA}
\scrA(\alpha) & := & \Ric( \alpha )+ \iota_{2D^*T + B}\alpha , \\ \nonumber
\scrD \alpha & := & \tr_{H} \nabla_\times \alpha(T(\times, \, \cdot \,)).
\end{eqnarray*}
\end{enumerate}
\end{lemma}
Notice the following identity $\scrD \alpha = - D^* \iota_T \alpha + \iota_{D^* T} \alpha$, while $\pr_H^* \scrD\alpha = S^* D \alpha$. We will use the Weitzenb\"och-type formula \eqref{Weitz} to find the spectral gap. However, we first need the following equalities.
\begin{lemma}\label{lemma:Inequalities}
Let $f \in C^\infty(M)$ be an arbitrary function.
\begin{enumerate}[\rm (a)]
\item We have the relation
$$\left\langle \iota_T df, D df \right\rangle_{g_H^{*\otimes 2}}= - \frac{1}{2}| \iota_T df|^{2}_{g_H^{*\otimes 2}}.$$
Hence, in particular, $| (D + \iota_T)df|_{L^2(g_H^{*\otimes 2})}^2 = | Ddf|_{L^2(g_H^{*\otimes 2})}^2.$ Furthermore,
\begin{equation} \label{CTver}
\frac{1}{2} \| \iota_T df \|^2_{L^2(g_H^{*\otimes 2})} - \langle \iota_B df, df \rangle_{L^2(g_H^{*\otimes 2})} 
= \langle (D + \iota_T) df, S(df) \rangle_{L^2(g_H^* \otimes g_V^*)}.
\end{equation}
\item We have the relation
$$ - \langle dLf, df \rangle_{L^2{(g_V^*})} = \| (D + \iota_T) df \|_{L^2(g_H^*\otimes g_V^*)}^2  - \langle \scrW df, df \rangle_{L^2(g_V^*)}.$$
\item Let $(Ddf)^s$ denote the symmetrized tensor $2(Ddf)^s(v,w) = \nabla_vdf (w) + \nabla_w df (v)$. Then
\begin{align*}
 \| L f\|_{L^2}^2  &= \| (D df)^s \|_{L^2(g_H^{*\otimes 2})}^2  - \frac{1}{2} \langle \iota_B (df), df \rangle_{L^2(g_H^*)}   \\
& \qquad + \langle \Ric(df), df \rangle_{L^2(g_H^*)} - \frac{3}{2} \langle (D + \iota_T) df, S(df) \rangle_{L^2(g_H^* \otimes g_V^*)}.
\end{align*}

\end{enumerate}
\end{lemma}
\begin{proof}
Observe that since $\nabla$ preserves $H$, we have that $D\pr_H^* \alpha= (\pr_H^* \otimes \pr_H^*) D\alpha$. Hence, for any one-form $\alpha$ and two-tensor $\xi$
$$\langle \alpha, D^*\xi \rangle_{L^2(g_H^*)} = \langle \pr_H^* \alpha, D^*\xi \rangle_{L^2(g^*)} = \langle D\alpha, \xi \rangle_{L^2(g_H^{* \otimes 2})}. $$
Similarly, $\langle \alpha, D^* \xi \rangle_{L^2(g_V^*)} = \langle D \alpha, \xi \rangle_{L^2(g_H^* \otimes g_V^*)}$.
\begin{enumerate}[\rm (a)]
\item If $X_1, \dots, X_n$ is a local orthonormal basis of $H$, then
\begin{align*} & \langle D df, \iota_T df \rangle_{g_H^{*\otimes 2}} = \sum_{r,s=1}^n df(T(X_r, X_s))\nabla^2_{X_r,X_s} f \\
& = \frac{1}{2} \sum_{r,s=1}^n df(T(X_r, X_s)) ( \nabla^2_{X_r, X_s}f - \nabla_{X_s, X_r}^2 f) \\
&  = - \frac{1}{2} \sum_{r,s=1}^n df(T(X_r, X_s))^2 = - \frac{1}{2}| \iota_T df|_{g_H^{*\otimes 2}}^2. \end{align*}
We can use this identity to find
\begin{align*} 
&\frac{1}{2} \| \iota_T df \|^2_{L^2(g_H^{*\otimes 2})} = - \langle D^* \iota_T df, df \rangle_{L^2(g_H^{*})} \\
& = - \langle \iota_{D^* T} df, df \rangle_{L^2(g_H^{*})} + \langle D df, S(df) \rangle_{L^2(g_H^* \otimes g_V^*)} \\
&=  \langle \iota_B df, df \rangle_{L^2(g_H^{*})} + \langle (D + \iota_T) df, S(df) \rangle_{L^2(g_H^* \otimes g_V^*)} .
\end{align*}

\item If we evaluate \eqref{Weitz} with $\sharp^V df$, we obtain
\begin{align*}
& \langle L df, df \rangle_{L^2(g_V^*)} - \langle dL f , df \rangle_{L^2(g_V^*)} = - \| D df\|_{L^2(g_H^* \otimes g_V^*)}^2 - \langle dL f , df \rangle_{L^2(g_V^*)}\\
& = 2 \langle \iota_T df, D df \rangle_{L^2(g_H^* \otimes g_V^*)} + \langle \iota_B df, df \rangle_{L^2(g_V^*)}.
\end{align*}
The result now follows from completing the square.
\item By evaluating \eqref{Weitz} with $\sharp^H df$ and doing similar computations as in (b), we obtain
\begin{align*}
& \| L f \|_{L^2}^2 =  - \langle dL f , df \rangle_{L^2(g_H^*)}  \\
& = \| (D+ \iota_T) df\|_{L^2(g_H^{*\otimes 2})}^2 -  \|\iota_T df \|_{L^2(g_H^{*\otimes 2})}^2 + \langle (\Ric +\iota_B)(df), df \rangle_{L^2(g_V^*)}.
\end{align*}
From (a), we know that $ \| (D+ \iota_T) df\|_{L^2(g_H^{*\otimes 2})}^2 =  \| D df\|_{L^2(g_H^{*\otimes 2})}^2$ and furthermore,
\begin{align*}
&| D df|_{g_H^{*\otimes 2}}^2 = | (D df)^s|_{g_H^{*\otimes 2}}^2 + \frac{1}{4} \sum_{r,s=1}^n (\nabla_{X_r,X_s}^2 f -\nabla_{X_s,X_r}^2 f)^2 \\
&=  | (D df)^s|_{g_H^{*\otimes 2}}^2 + \frac{1}{4} \sum_{r,s=1}^n (df T(X_r, X_s))^2 = | (D df)^s|_{g_H^{*\otimes 2}}^2 + \frac{1}{4} | \iota_T df |_{g_H^{* \otimes 2}}^2.
 \end{align*}
Applying \eqref{CTver} to
$$\| L f \|_{L^2}^2 = \| (D df)^s\|_{L^2(g_H^{*\otimes 2})}^2 -  \frac{3}{4}\|\iota_T df \|_{L^2(g_H^{*\otimes 2})}^2 + \langle (\Ric + \iota_B )(df), df \rangle_{L^2(g_V^*)}$$
we get the result.
\end{enumerate}
\end{proof}
\subsection{Proof of Theorem~\ref{th:main}}
Introduce the eigenfunction $f$ with eigenvalue $\lambda < 0$. We normalize $f$ such that $\| f \|_{L^2} = 1$ and hence $\| df \|_{L^2(g_H^*)}^2 = - \lambda$. Assume the bounds of Theorem~\ref{th:main}.
\begin{lemma} We will have the following bounds involving the vertical part of the gradient;
\begin{eqnarray*}
\rho_2 \| df \|_{L^2(g_V^*)} &\leq & 2\left(\sqrt{\kappa_1(\kappa_3 - \lambda)} + 2\kappa_2 \right) \sqrt{-\lambda} \\
| \langle (D + \iota_T) df, S(df) \rangle_{L^2(g_H^* \otimes g_V^*)} |& \leq & - \frac{2 \lambda }{\rho_2} \sqrt{\kappa_1(\kappa_3 - \lambda)} \left(\sqrt{\kappa_1(\kappa_3 - \lambda)} + 2\kappa_2 \right)  .
\end{eqnarray*}
\end{lemma}
\begin{proof}
Using Lemma~\ref{lemma:Inequalities}~(b), we obtain the equality
$$\langle \scrW df, df \rangle_{L^2(g_V^*)} - \lambda \|  df \|_{L^2{(g_V^*})}^2 = \| (D + \iota_T) df \|_{L^2(g_H^*\otimes g_V^*)}^2 ,$$
which implies that $\| (D+ \iota_T)df\|_{L^2(g_H^*\otimes g_V^*)}^2 \leq (\kappa_3 - \lambda) \| df \|_{L^2( g_V^*)}^2.$ The bounds then follow from the estimate
\begin{align*} 
& \frac{\rho_2}{2} \| df \|_{L^2(g_V^*)}^2 -2\kappa_2 \| df \|_{L^2(g_V^*)}  \| df \|_{L^2(g_H^*)}  \leq \frac{1}{2} \| \iota_T df \|^2_{L^2(g_H^{*\otimes 2})} - \langle \iota_B df, df \rangle_{L^2(g_H^*)} \\
& =  \langle (D + \iota_T) df, S(df) \rangle_{L^2(g_H^* \otimes g_V^*)} \leq \sqrt{\kappa_1(\kappa_3 - \lambda)}   \| df\|_{L^2(g_H^*)} \| df\|_{L^2(g_V^*)} .
\end{align*}
\end{proof}

We are now ready to complete the proof of Theorem~\ref{th:main}. Observe from Lemma~\ref{lemma:Inequalities}~(c) that
\begin{align*}
& \lambda^2 = \| L f\|_{L^2}^2    = \| (D df)^s \|_{L^2(g_H^{*\otimes 2})}^2  - \frac{1}{2} \langle \iota_B df, df \rangle_{L^2(g_H^*)} \\
& \qquad \qquad\qquad  + \langle \Ric(df), df \rangle_{L^2(g_H^*)} - \frac{3}{2} \langle (D + \iota_T) df, S(df) \rangle_{L^2(g_H^* \otimes g_V^*)} \\
& \geq \frac{1}{n} \|Lf \|_{L^2}^2 - \kappa_2 \| df \|_{L^2(g_V^*)} \| df \|_{L^2(g_H^*)}  + \rho_1 \| df \|_{L^2(g_H^*)}^2 \\
& \qquad \qquad \qquad \qquad + \frac{3}{2} \frac{2 \lambda }{\rho_2} \sqrt{\kappa_1(\kappa_3 - \lambda)} \left(\sqrt{\kappa_1(\kappa_3 - \lambda)} + 2\kappa_2 \right) \\
& \geq \frac{\lambda^2}{n} + \frac{2\kappa_2 \lambda}{\rho_2}  \left(\sqrt{\kappa_1(\kappa_3 - \lambda)} + 2\kappa_2 \right)   - \rho_1 \lambda \\
& \qquad \qquad \qquad \qquad +  \frac{3 \lambda }{\rho_2} \sqrt{\kappa_1(\kappa_3 - \lambda)} \left(\sqrt{\kappa_1(\kappa_3 - \lambda)} + 2\kappa_2 \right) .
\end{align*}
This gives us the inequality
\begin{align*}
- \frac{n-1}{n} \rho_2 \lambda  & \geq \rho_1 \rho_2 - \left(3 \sqrt{\kappa_1(\kappa_3 - \lambda)} + 2\kappa_2 \right)  \left(\sqrt{\kappa_1(\kappa_3 - \lambda)} + 2\kappa_2 \right) .
\end{align*}
Define $s = \sqrt{\kappa_1(\kappa_3 - \lambda)}$. Then
\begin{align*}
\frac{n-1}{n} \left(\frac{\rho_2}{\kappa_1} s^2  - \rho_2 \kappa_3\right) & \geq \rho_1 \rho_2 - 3s^2 - 8 s \kappa_2 - 4\kappa_2^2
\end{align*}
or
\begin{align*}
\left( \frac{n-1}{n}\frac{\rho_2}{\kappa_1} + 3 \right) s^2 + 8 s \kappa_2- \left(\rho_1 \rho_2 - 4\kappa_2^2  + \frac{n-1}{n} \rho_2 \kappa_3 \right) & \geq 0 .
\end{align*}
From this inequality, we have
\begin{align*}
s \geq & \sqrt{\frac{\rho_1 \rho_2 - 4\kappa_2^2  + \frac{n-1}{n} \rho_2 \kappa_3}{\frac{n-1}{n}\frac{\rho_2}{\kappa_1} + 3} + \left( \frac{4 \kappa_2}{\frac{n-1}{n}\frac{\rho_2}{\kappa_1} + 3} \right)^2 } -\frac{4 \kappa_2}{\frac{n-1}{n}\frac{\rho_2}{\kappa_1} + 3}.
\end{align*}
The result follows.

\section{Examples} \label{sec:Examples}

\subsection{Example with non-integrable orthogonal complement}\label{ex:non-integrable}
Consider the Lie group $\SO(4)$ with Lie algebra $\mathfrak{so}(4)$, the latter consisting of all $4\times4$ skew-symmetric matrices. This Lie algebra is spanned by matrices $B^{ij}=e_{i}e_{j}^t-e_{j}e_{i}^t$, where $e_1$, $e_2$, $e_3$, $e_4$ is the standard basis of $\mathbb{R}^4$. By abuse of notation, we will use the same symbol to denote an element of the Lie algebra and the vector field on $\SO(4)$ obtained by left translation. Consider the horizontal bundle
$$H = \Span \{ X_1, X_2, X_3, X_4 \} := \Span \{ B^{12}, B^{14}, B^{24}, B^{34} \},$$
and define a metric $g_H$ such that $X_1, X_2, X_3, X_4$ is an orthonormal basis.
We will show that the first eigenvalue $\lambda_1$ of the operator $$L = X_1^2 + X_2^2 + X_3^2 + X_4^2$$ is bounded by
$$\lambda_1 \le -\frac{8}{51}.$$

Define an inner product on the Lie algebra by $\left\langle A, B \right\rangle_{\mathfrak{so}(4)} =  - \frac{1}{2} \tr AB$ and extend this to a Riemannian metric $g$ by left translation. Observe that $g$ is a bi-invariant metric on $\SO(4)$ which tames $g_H$. Furthermore, if $\mu$ is the volume form of $g$, then $L = \Delta_{H,\mu}$.
The vertical bundle $$V=\Span\{Z_1, Z_2\} : = \Span \{ B^{13}, B^{23}\},$$
is not integrable, but properties \eqref{item:A}, \eqref{item:B} and \eqref{item:C} are still satisfied.
From bi-invariance, we also have that \eqref{TGF} holds, so we can use the Bott connection \eqref{Bott} from Remark~\ref{re:PostMain}.
The bracket relations, connection and torsion are given by
$$\resizebox{.9\hsize}{!}{$ \begin{array}{l|ccccccc}
\left[ X,Y\right] &X_1&X_2&X_3&X_4&Z_1&Z_2\\ \hline
X_1 & 0& - X_3& X_2 & 0 & - Z_2& Z_1\\
X_2 & X_3 & 0 & - X_1 &- Z_1 & X_4 &0\\
X_3 & - X_2&  X_1 & 0 &- Z_2& 0 & X_4\\
X_4 & 0 & Z_1 & Z_2 & 0 & -X_2 & -X_3 \\
Z_1 & Z_2 & - X_4 & 0 & X_2 & 0 &- X_1\\
Z_2& - Z_1 & 0 & - X_4 & X_3 & X_1 & 0
\end{array}\qquad  \begin{array}{l|ccccccc}
\nabla_{X} Y & X_1 & X_2 & X_3 & X_4 & Z_1 & Z_2 \\ \hline
X_1 & 0 & -\frac{1}{2} X_3 & \frac{1}{2} X_2 & 0 & - Z_2 & Z_1\\
X_2 & \frac{1}{2} X_3 & 0 & -\frac{1}{2} X_1 & 0 & 0 & 0 \\
X_3 & - \frac{1}{2} X_2 & \frac{1}{2} X_1 & 0 & 0 & 0 & 0 \\
X_4 & 0 & 0 & 0 & 0 & 0 & 0 \\
Z_1 & 0 & -X_4 & 0 & X_2 & 0 & 0 \\
Z_2 & 0 & 0 & - X_4& X_3& 0 & 0
\end{array}$}$$
$$\resizebox{.45\hsize}{!}{$ \begin{array}{l|ccccccc}
T( X,Y ) & X_1 & X_2 & X_3 & X_4 & Z_1 & Z_2 \\ \hline
X_1 & 0 & 0 & 0 & 0 & 0 & 0 \\
X_2 & 0 & 0 & 0 & Z_1 & 0 & 0 \\
X_3 & 0 & 0 & 0 & Z_2 & 0 & 0 \\
X_4 & 0 & - Z_1 & - Z_2 & 0 & 0 & 0 \\
Z_1 & 0 & 0 & 0 & 0 & 0 & X_1 \\
Z_2 & 0 & 0 & 0 & 0 & -X_1 & 0
\end{array}$}$$
Additionally, $\Ric$ is given by
\[ \begin{array}{l|ccccc}
\Ric(\flat^g X)(Y) & X_1 & X_2 & X_3 & X_4 \\ \hline
X_1 & 1/2 & 0 & 0 & 0 \\
X_2 & 0 & 3/2 & 0 & 0 \\
X_3 & 0 & 0 & 3/2 & 0 \\
X_4 & 0 & 0 & 0 & 2
\end{array}\]
hence $\rho_1= 1/2$. Using the tables above we have $B = 0$ and $\scrW = 0$, hence $\kappa_2=\kappa_3=0.$
The last two constants can be set to $\rho_2=\kappa_1=2,$ since $$|\iota_{T}\alpha|^{2}_{ {g_H^*}^{\otimes 2}}= 2 (\alpha(Z_1)^2 + \alpha(Z_2)^2) = 2 | \alpha|_{g_V^*}^2,$$
$$|S(\alpha )|^{2}_{g_H^*\otimes g_{V}^*}=  \alpha(X_2)^2+ \alpha(X_3)^2+ 2\alpha(X_4)^2 \le 2|\alpha|^{2}_{g_H^*}.$$
Using Theorem~\ref{th:main} with the constants $\rho_1=\frac{1}{2}$, $\rho_2=\kappa_1=2$ and $\kappa_2=\kappa_3=0$ we obtain the bound.

\subsection{Example with conformal change of metric and the Yang-Mills condition} \label{sec:YM}
For the second example consider the Lie group $$\SU( 2 )=\left\{ a\in M_{2\times 2}( \mathbb{C} ):a=\left( \begin{array}{cc}z_1&z_2\\-\bar{z_2}&\bar{z_1}\end{array} \right), \det( a )=1 \right\}.$$
Its Lie algebra $\mathfrak{su}(2)$ is spanned by the matrices
   \[X=\frac{1}{\sqrt{2}}\left( \begin{array}{cc}
      0&i\\i&0
   \end{array}\right),\quad Y=\frac{1}{\sqrt{2}}\left( \begin{array}{cc}
      0&-1\\1&0
   \end{array}\right)\quad\text{and}\quad Z=\frac{1}{\sqrt{2}}\left( \begin{array}{cc}
      i&0\\0&-i
   \end{array}\right).\]
   Let $g_H$ be a sub-Riemannian metric on the subspace spanned by $X$ and $Y$ making the vector fields orthonormal.
   For any smooth function $f$ satisfying $Z( f )=0$, define a new sub-Riemannian metric by $g_{f,H} =e^{2f} g_H$. Then $X_f=e^{-f}X$ and $Y_f=e^{-f}Y$ is an orthonormal frame for $g_{f,H}$. Let $g_{f}$ be the taming metric making $X_f,Y_f,Z$ orthonormal. Define the connection $\nabla$ as in equation~\eqref{Bott} with respect to the taming metric $g_f$. In this setting we have the sub-Laplacian $$L_f =e^{-2f}(X^2+Y^2 ) = : e^{-2f} L.$$
The bracket relations, connection and torsion are given by
$$[X_f,Y_f] = e^{-2f}Z-X_f( f )Y_f+Y_f( f )X_f, \quad [Y_f, Z] = X_f, \quad [Z, X_f] = Y_f,$$
$$\resizebox{.8\hsize}{!}{ $\begin{array}{l|ccc}
\nabla_{X}Y &X_f&Y_f&Z\\ \hline
X_f & - (Y_f  f )Y_f & (Y_{f} f )X_f & 0 \\
Y_f & (X_{f} f ) Y_f & -(X_{f} f ) X_f & 0 \\
Z & Y_{f} & -X_f & 0 
\end{array}\qquad 
\begin{array}{l|ccc}
T( X,Y ) & X_f & Y_f & Z \\ \hline
X_f & 0 & -e^{-2f}Z & 0 \\
Y_f & e^{-2f}Z & 0 & 0\\
Z & 0 & 0 & 0 \end{array}$}$$
The constant $\rho_1$ is
$$\rho_1=\inf_{a\in \SU( 2 )}e^{-2f(a)}( 1-Lf(a) )$$ 
since the only non-zero components of the curvature are
$$\langle R(Y_f,X_f) X_f, Y_f \rangle_{g_f} = \langle R(X_f,Y_f) Y_f, X_f \rangle_{g_f} = e^{-2f}( 1-Lf ).$$
Observe that $$\scrW =0,\quad | \iota_T \alpha|^2_{g_{f,H}^{*\otimes 2}} = 2 e^{-4f} | \alpha|^2_{g_V^*},\quad \text{and}
\quad| S(\alpha) |_{g_{f,H}^* \otimes g_V^*}^2 = e^{-4f} | \alpha |_{g_{f,H}^*}^2.$$
Hence we can choose
$$\rho_2 = 2 \inf_{a \in \SU(2)} e^{-4f(a)}, \quad \kappa_1 = \sup_{a \in \SU(2)} e^{-4f(a)}.$$
Finally, we have
$$B: X_f \mapsto -2 e^{-2f} (Y_f f) Z, \qquad Y_f \mapsto 2 e^{-2f} (X_f f) Z,$$
so
\begin{align*} \langle \iota_B \alpha, \alpha\rangle_{g_{f,H}^*} &= 2 e^{-2f} \alpha(Z) \langle -  (Y_f f) X_f + (X_f f ) Y_f, \sharp^{f,H} \alpha \rangle_{g_{f,H}} \\
& \leq 2e^{-2f} | df |_{g_{f,H}^*} |\alpha|_{g_V^*}  |\alpha|_{g_{f,H}^*}, \end{align*}
and we can choose
$$\kappa_2 = \sup_{a \in \SU(2)} e^{-3f} | df |_{g_H^*}.$$
Inserting this in the formula of Theorem~\ref{th:main}, we have the spectral gap estimate.

To give a more concrete formula, define $m = \inf_{a \in \SU(2)} f(a)$ and $m + M = \sup_{a \in \SU(2)} f(a)$. Using the inequalities $$\rho_1 \leq e^{-2(m+M)} (1 - \| Lf\|_{L^\infty})\quad \text{and} \quad \kappa_2 \leq e^{-3m} \| df \|_{L^\infty(g_H^*)},$$ we obtain
$$\textstyle-\lambda_{1}\ge e^{-2m} \left( \sqrt{\frac{ 2 e^{-6M} (1 - \| Lf \|_{L^\infty} )  - 4  \| df \|^2_{L^\infty( g^* )}}{e^{ -4M} + 3 } + \left( \frac{4 \| df\|_{L^\infty(g^*)}  }{ e^{ - 4M} + 3 } \right)^2 } -\frac{4 \| df \|_{L^\infty(g^*)} }{ e^{- 4M}  + 3 } \right)^{2}.$$
For the special case
$$f(a) = c |z_1|^2, \qquad c >0,$$
we have $m = 0$ and $M = c$. Furthermore, we have
$$X( f )=\frac{ci}{\sqrt{2}}( z_2\bar{z}_1-\bar{z}_2z_1 ) = \sqrt{2} c \mathrm{Im}(z_1 \bar{z}_2),\quad Y( f )=\frac{c}{\sqrt{2}}( z_2\bar{z}_1+\bar{z}_2z_1 ) = \sqrt{2} c \mathrm{Re}(z_1 \bar{z}_2)$$ and $$X( X( f ) )=Y( Y( f ) )=c( |z_2|^{2}-|z_1|^{2} ).$$
Hence
$$| df |_{g_H^*}^2 = 2c^2 |z_1|^2 |z_2|^2, \qquad Lf = 2c (|z_2|^2 - |z_1|^2),$$
so $\| df \|_{L^\infty(g_H^*)} = \| Lf \|_{L^\infty} = 2c$. In conclusion
$$-\lambda_{1}\ge \left(\sqrt{\frac{ 2 e^{-6c} (1 - 2c )  - 16  c^2  }{ e^{ -4c} + 3 } + \left( \frac{8c}{e^{ - 4c} + 3 } \right)^2 } -\frac{8 c }{ e^{- 4c}  + 3 } \right)^{2}.$$
This estimate is non-trivial whenever $e^{-6c} (1-2c) > 8c^2$. In particular, this is true when $c \in [0, 1.17139]$.

\subsection{Example with a non-totally geodesic foliation}
Consider the Lie group $\SU(2)\times \SU(2)$ with its Lie algebra $\mathfrak{su}(2)\times \mathfrak{su}(2)$. Let $X_j, Y_j, Z_j$, $j =1,2$ be the matrices given in the previous example for each copy. Define a new frame by
$$X^\pm = X_1 \pm X_2, \quad Y^\pm = Y_1 \pm Y_2, \quad Z^\pm = Z_1 \pm Z_2.$$
For any real number $c\in \mathbb{R}$, define
$$X^c = X^- + c X^+.$$
Let $H^c$ be the span of $X^c,Y^-$ and $Z^-$, with the corresponding metric $g_{H^c}$ making the vector fields $X^c,Y^-$ and $Z^-$ orthonormal. Define the taming metric $g_c$ by the vector fields $X^c$, $Y^-$, $Z^-$, $X^+$, $Y^+$ and $Z^+$ forming an orthonormal frame. 
Denote by $\mu$ the volume density of $g_c$, and note that it is independent of $c$. The sub-Laplacian then becomes
$$L_c = \Delta_{H^c, \mu} = (X^c)^2 + (Y^-)^2 + (Z^-)^2.$$
For this basis the bracket relations become
$$\resizebox{.8\hsize}{!}{$\begin{array}{c|cccccc}
{[A,B]} & X^c & Y^- & Z^- & X^+ & Y^+ & Z^+ \\ \hline
X^c &0&Z^++cZ^-&-Y^+-cY^-&0&Z^-+cZ^+&-Y^--cY^+ \\
Y^- & - Z^+-cZ^- & 0 & X^+ & -Z^- & 0 & X^c-cX^+\\
Z^- &Y^++cY^-&-X^+&0&Y^-&cX^+-X^c&0 \\
X^+ &0&Z^-&-Y^-&0&-Z^+&-Y^+ \\
Y^+ &-Z^--cZ^+&0&X^c-cX^+&Z^+&0&X^+ \\
Z^+ & Y^-+cY^+&cX^+-X^c&0&Y^+&-X^+&0
\end{array}$}$$ The complement $V=\Span\{X^+,Y^+,Z^+\}$ is integrable and metric preserving, hence assumption \eqref{item:A} and \eqref{item:B} are satisfied. It is also straightforward to check that $\tr_{{H^c}}( \calL_{A}g_c )(\times,\times)=0$ whenever $A$ is vertical.
Note that we can not use the Bott connection for $c \neq 0$ since $$(\calL_{Y^-}\pr_{V}^{*}g_V)(r X^+ +s Z^+,r X^++s Z^+)=-c rs , \qquad r,s \in \mathbb{R}^n.$$
Define the connection $\nabla$ by $$\nabla_{X}Y=\pr_{H^c}\nabla^{g_c}_{\pr_{H^c}X}\pr_{H^c}Y+\pr_{H^c}\left[\pr_{V}X,\pr_{H^c}X\right]+\pr_{V}\nabla'_{\pr_{V}X}\pr_{V}Y,$$
where $\nabla'$ is defined by making $X^+,Y^+$ and $Z^+$ into a parallel frame. The covariant derivative and torsion are 
$$\resizebox{.9\hsize}{!}{$\begin{array}{c|cccccc}
{\nabla_AB} & X^c & Y^- & Z^- & X^+ & Y^+ & Z^+ \\ \hline
X^c &0&cZ^-&-cY^-&0&0&0 \\
Y^- & 0&0&0&0&0&0\\
Z^- & 0&0&0&0&0&0\\
X^+ &0&Z^-&-Y^-&0&0&0\\
Y^+ &-Z^-&0&X^c&0&0&0\\
Z^+ & Y^-&-X^c&0&0&0&0
\end{array} \qquad \begin{array}{c|cccccc}
   {T(A,B)} & X^c & Y^- & Z^- & X^+ & Y^+ & Z^+ \\ \hline
X^c &0&-Z^+&Y^+&0&-cZ^+&cY^+ \\
Y^- & Z^+&0&-X^+&0&0&cX^+\\
Z^- &-Y^+&X^+&0&0&-cX^+&0 \\
X^+ &0&0&0&0&Z^+&Y^+ \\
Y^+ &cZ^+&0&cX^+&-Z^+&0&-X^+ \\
Z^+ & -cY^+&-cX^+&0&-Y^+&X^+&0
\end{array}$}$$
The only non-zero curvature terms that affect the tensor $\Ric$ are given by
\begin{align*} & 1=  \langle R^{\nabla}(Y^-, X^c) X^c, Y^- \rangle_{g_{H^c}} = \langle R^{\nabla}(Z^-, X^c) X^c, Z^- \rangle_{g_{H^c}}  = \langle R^{\nabla}(X^c, Y^-) Y^-, X^c \rangle_{g_{H^c}} \\
& = \langle R^{\nabla}(Z^-, Y^-) Y^-, Z^- \rangle_{g_{H^c}} = \langle R^{\nabla}(X^c, Z^-) Z^-, X^c \rangle_{g_{H^c}} = \langle R^{\nabla}(Y^-, Z^-) Z^-, Y^- \rangle_{g_{H^c}} .\end{align*}
Hence $\rho_1=2$. Furthermore, $|\iota_T \alpha|^2_{g_{H^c}^{*\otimes 2}} =2 | \alpha |^2_{g_V^*}$, while $|S(\alpha) |^2_{g_{H^c}^* \otimes g_V^*} = 2 |\alpha|^2_{g_{H^c}^*}$, so we can choose $\rho_2=\kappa_1=2$. Next, we have
$$B: X^c \mapsto - 2 cX^+, \quad Y^- \mapsto 0, \qquad Z^- \mapsto 0,$$
and we can set $\kappa_2 = |c|$.  
To calculate $\kappa_3$, we see that $$\alpha( \tr_{H^c} T( \times,T( \times,\sharp^{g_{V}}\alpha ) ))= - c^{2}( \alpha(Y^+)^2+\alpha(Z^+)^{2} ),$$
$$|\iota_T\alpha |^{2}_{ {g_{H^c}}^* \otimes g_V^*}=c^{2}(2\alpha(X^+)^2 +  \alpha(Y^+)^{2}+\alpha(Z^+)^2 ),$$ and $$\alpha( ( \tr_{H^c}\nabla_{\times}T )( \times,\sharp^{g_{V}}\alpha ) )=0.$$ In conclusion, we have $\langle \scrW(\alpha), \alpha \rangle_{g_V^*} = 2c^2 |\alpha|^2_{g_V}$ leading to $\kappa_3=2c^{2}$.  
Using Theorem~\ref{th:main} we have that
$$-\lambda_{1} \ge \frac{2}{121} \left(\sqrt{33 + 25 c^2 } -6 |c|  \right)^{2}- 2c^2.$$
This estimate on the spectral gap is non-trivial when $|c| <  \frac{1}{4} \sqrt{\frac{11}{123}}.$

\bibliographystyle{amsplain}


\end{document}